\documentclass[a4paper,11pt]{amsart}

\usepackage{epsfig}
\usepackage{amsthm,amsfonts}
\usepackage{amssymb,graphicx,color}
\usepackage[all]{xy}
\usepackage{verbatim}
\usepackage{hyperref}
\usepackage{enumitem}

\newtheorem{theorem}{Theorem}[section]
\newtheorem*{theorem*}{Theorem}

\newtheorem{definition}[theorem]{Definition}

\newtheorem{remark}[theorem]{Remark}

\newcommand{\C}{\mathbb{C}}

\begin{document}

\title[Manifold with infinitely many fibrations over the sphere]
{Manifold with infinitely many fibrations over the sphere}

\author[W.Jelonek]{W\l odzimierz Jelonek}

\author[Z. Jelonek]{Zbigniew Jelonek}

{\address[W{\l}odzimierz Jelonek]{ Katedra Matematyki Stosowanej, Politechnika Krakowska, Warszawska 24, 31-155 Krak\'ow, Poland. \newline
              E-mail: {\tt wjelon@pk.edu.pl}}

{\address[Zbigniew Jelonek]{ Instytut Matematyczny, Polska Akademia Nauk, \'Sniadeckich 8, 00-656 Warszawa, Poland. \newline
              E-mail: {\tt najelone@cyf-kr.edu.pl}}


\begin{abstract}
We show that the manifold $X=S^2\times S^3$ has infinitely many structures of a fiber bundle over the base $B=S^2.$ In fact for every lens space $L(p,1)$ there is a fiber bundle $L(p,1)\to X\to B.$  \end{abstract}

\maketitle

\section{Introduction}

Let $X,F,B$ be smooth  manifolds, where $X,B$ are connected. Consider  smooth fibrations $F\to X\to B.$ If  $X=B=S^1$, then  there are infinitely many different fibers $F_k$ such that there exists a fibration $F_k\to X\to B.$
Tollefson in \cite{tlf} found some Seifert manifolds having infinitely many different fiber bundle structures over $S^1$. 
All Tollefson's examples  have connected fibers. W.P. Thurston in \cite{thu} showed that if a hyperbolic 3-manifold with $b_1>1$ fibers over $S^1$, then it fibers in infinitely many ways.
Moreover, if dimension of $F\ge 5$, we can construct such fiber bundles over other odd dimensional spheres using Milnor constructions of infinite family of $h-$cobordant manifolds (see  \cite{mil}).

However  to the best knowledge of the authors there were no  examples of  manifolds with infinitely many different structures of a fiber bundle (with different fibers)  over  even dimensional spheres.
Here using ideas from  \cite{gib}, \cite{Bb}  and \cite{j}  we show the following:

\vspace{5mm}

{\bf Theorem 3.5.}
 {\it The manifold $X=S^2\times S^3$ has infinitely many structures of a fiber bundle over the base $B=S^2.$ In fact, for every lens space $L(p,1)$ there is a fiber bundle:  $L(p,1)\to X\to B,$  $p=1,2,....$ }

 \section{Preliminaries}\label{section:preliminaries}

We start with two definitions.

\begin{definition}
Let $X\subset \Bbb {CP}^n$ be an algebraic variety. We assume $\mathbb {CP}^n$  to be a hyperplane at infinity of $\mathbb {CP}^{n+1}$.  Then by an algebraic cone $\overline{C(X)}\subset \Bbb {CP}^{n+1}$ with base $X$ we mean the set
$$\overline{C(X)}=\bigcup_{x\in X} \overline{O,x},$$ where $O$ is the center of coordinates in $\mathbb C^{n+1}\subset \mathbb {CP}^{n+1}$, and $\overline{O,x}$ means the projective line which goes through $O$ and $x.$ By an affine cone $C(X)$ we mean $\overline{C(X)}\setminus X.$
By the link of $C(X)$ we mean the set $L=\{ x\in C(X): ||x||=1\}.$
\end{definition}

\begin{definition}
The three-dimensional lens spaces $L(p;q)$ are quotients of the sphere $S^{3}$ by $\mathbb {Z} /p$-actions. More precisely, let $p$ and $q$ be coprime integers and consider $S^{3}$ as the unit sphere in $\C^2$.
Then the $\mathbb {Z} /p$-action on $S^{3}$ generated by the homeomorphism

    $$ (z_1 , z_2 ) \mapsto ( e^{ 2 \pi i / p} z_1 , e^{ 2 \pi i q / p} z_2 )$$

is free. The resulting quotient space is called the lens space $L(p;q).$
\end{definition}

\vspace{5mm}

It is well known that

$$H_k(L(p,q),\mathbb{Z}) = \left\lbrace \begin{array}{rl} \mathbb{Z}, & \qquad k = 0,3 \\ \mathbb{Z}/p\mathbb{Z}, & \qquad k = 1 \\ 0, & \qquad \operatorname{otherwise}\\\end{array}\right. $$

\vspace{5mm}

\section{Main Result}

Let $W_k$ denote the Veronese  embedding of degree $k$ of $\C\Bbb P^1$ into $\C\Bbb P^k$ given by   $\psi([z_0,z_1])=[z_0^k:z_0^{k-1}z_1:....,z_0z_1^{k-1}:z_1^k]$. Let $k\ge 1$ and consider the varieties $X_{k}=\phi( W_k\times \C\Bbb P^{1})\subset \C\Bbb P^{n_k}$, where $\phi:\C\Bbb P^k\times \C\Bbb P^1\rightarrow \C\Bbb P^{n_k},  \phi([z_0:z_1:...:z_k],[w_0,w_1])=[z_0w_0:z_0w_1:z_1w_0:z_1w_1,...:z_kw_0:z_kw_1]$ is the Segre embedding (here $n_k=2k+1)$.  Consider the affine cone $C(X_k).$
Let $L_k$ be the link of this cone. We have

\begin{theorem}\label{1}
All manifolds $L_{k}$ are diffeomorphic to $S^2\times S^3=X.$
\end{theorem}

\begin{proof}

We use here the following Theorem  of P. J. Giblin (see \cite{gib}):

\vspace {2mm}

\begin{theorem}
The total space of a circle bundle over $S^2\times S^2$ is diffeomorphic to $S^2\times S^3$  if and only if it is simply connected.
\end{theorem}

\begin{remark}
{\rm P.J.  Giblin formulated his result in the topological category only, but it is easy to see that his proof works in a smooth case also.}
\end{remark}

Now let $W_k$ denote the Veronese  embedding of degree $k$ of $\C\Bbb P^1$ into $\C\Bbb P^k.$ Let $
k\ge 1$ and consider the varieties $X_{k}=\phi( W_k\times \C\Bbb P^{1})\subset \C\Bbb P^{n_k}$, where $\phi$ is the Segre embedding. Consider the affine cone $C(X_k).$
Let $L_k$ be the link of this cone.
By construction, $X_{k}$ is the  union of projective lines $X_k=\bigcup_{a\in  W_k}  \phi(\{a\} \times \Bbb P^1)$. This means that the affine cone ${C(X_{k})}$ is the  union of planes
$\bigcup_{a\in W_k}   C(\phi(\{a\} \times \Bbb P^1)).$ Hence the link $L_{k}$ of this cone is a union of  spheres $S^3.$ In fact using the Ehresmann Theorem, it is easy to observe that  this link is a  bundle over $W_k\cong S^2$ with the projection being the  composition of the projection $p: \C\Bbb P^{n_k+1} \setminus \{0\}\to \C\Bbb P^{n_k}$  and the projection $q:  W_k\times\C \Bbb P^{1}\to  W_k.$ Consequently $L_k$ is a $S^3$ bundle over $S^2.$  By a homotopy sequence
$$0=\pi_1(S^3)\to L_k\to \pi_1(S^2)\to 0$$ we get $\pi_1(L_k)=0$. Since $L_k$ is the (Hopf) circle bundle over $X_k\cong S^2\times S^2$  our result follows from the Giblin theorem.
\end{proof}

\begin{theorem}\label{2}
Let $W\subset \C\Bbb P^n$ be a smooth rational curve   of degree $d.$ Let $C(W)$ be an affine cone with base $W.$ Then the link $L$ of this cone at $0$ is diffeomorphic to the lens space $L(d,1).$
\end{theorem}

\begin{proof}
Note that $L$ is a principal circle bundle over the sphere. By \cite{blair} (see Theorem 2.1, Theorem 2.2 and Theorem 2.3) the group of such bundles is cyclic and generated by $S^3=L(1,1).$ Moreover, $pS^3=S^3/G_p$, where $G_p=\Bbb Z/(p)$. Hence $L$ is diffeomorphic to some $pS^3$, which  is the lens space $L(p,1)$. By \cite{Bb}, Thm. 3.5  we have that the torsion part of $H^2(L, \Bbb Z)$ is equal to $\Bbb Z/(d).$ On the other hand by \cite{kol}, point 11,  the space $L$ is a rational homology $3$-sphere, in particular $H_1(L,\Bbb Z)$ is a torsion group. Since the torsion part of $H_1(L,\Bbb Z)$ coincides with the torsion part of $H^2(L,\Bbb Z)$ we have $H_1(L,\Bbb Z)=\Bbb Z/(d).$ Hence $p=d$ and $L=L(d,1).$
\end{proof}

\begin{theorem}
 The manifold $X=S^2\times S^3$ has infinitely many structures of a fiber bundle over the base $B=S^2.$ In fact for every lens space $L(p,1)$ there is a fiber bundle:  $L(p,1)\to X\to B,$  $p=1,2,....$
\end{theorem}

\begin{proof}
Let $W_k$ denote the Veronese  embedding of degree $k$ of $\C\Bbb P^1$ into $\C\Bbb P^k.$ Let $
k\ge 1$ and consider the varieties $X_{k}=\phi( W_k\times \C\Bbb P^{1})\subset \C\Bbb P^{n_k}$, where $\phi$ is the Segre embedding. Consider the affine cone $C(X_k).$
Let $L_k$ be the link of this cone.
By construction, $X_{k}$ is the  union of projective rational curves $X_k=\bigcup_{a\in \C\Bbb P^1} \phi(W_k \times \{a\})$. This means that the affine cone ${C(X_{k})}$ is the  union of cones
$\bigcup_{a\in \C\Bbb P^1}  C(\phi(W_k \times \{a\})).$ Thus by Theorem \ref{2} the link $L_{k}$ of this cone is a union of lens spaces $L(k,1).$ In fact using the Ehresmann Theorem, it is easy to observe that  this link is a  bundle over $\C\Bbb P^1\cong S^2$ with the projection being the  composition of the projection $p: \C\Bbb P^{n_k+1} \setminus \{0\}\to \C\Bbb P^{n_k}$  and the projection $q:  W_k\times\C \Bbb P^{1}\to \C\Bbb P^1.$ Consequently $L_k$ is a $L(k,1)$ bundle over $S^2.$ By Theorem \ref{1}
this finishes the proof.\end{proof}

\vspace{5mm}

Manuscript does not contain any data.


\begin{thebibliography}{99}




\bibitem{blair} Blair, D.E.,
Riemannian Geometry of Contact and Symplectic Manifolds,
Second Edition, Progress in Mathematics
Volume 203, Birkh\"auser (2010).





\bibitem{Bb} Bobadilla, J.F. de; Fernandes, A. and Sampaio, J. E.
{\it Multiplicity and degree as bi-lipschitz invariants for complex sets}.
Journal of Topology, vol. 11 (2018), 958-966.


\bibitem{j} Fernandes, A., Jelonek Z., Sampaio J. E., {\it Bi-Lipschitz  equivalent cones with  different degrees}, arXiv:2309.07078.


\bibitem{gib} Giblin, P. J., {\it Circle bundles over a complex quadric},  Journal of the London Mathematical Society, Volume s1-43, Issue 1, 1968, 323–324.








\bibitem{kol} Kollar, J., 
{\it Links of complex analytic singularities}.  In: Surveys in Differential Geometry XVIII.
International Press of Boston, (2013), 157--192.

\bibitem{mil}   Milnor. J., {\em Whitehead torsion}, 
Bull. Amer. Math. Soc. 72(3), 358-426 (1966). 


\bibitem{thu} Thurston. W, {\it  A norm for the homology of 3-manifolds}, Memoirs of the American Mathematical Society, 59(339):99–130, 1986.


\bibitem{tlf} Tollefson. J.,  {\it $3$-manifolds fibering over $S^1$ with nonunique connected fiber}, Proc. Amer. Math.Soc. 21 (1969), 79–80.
 







\end{thebibliography}
\end{document}